\documentclass[12pt]{amsart}

\usepackage{amsfonts}
\usepackage{amsmath}
\usepackage{amssymb}
\usepackage{latexsym}

\def\Ito{It\^{o} }
\def\Holder{H\"{o}lder }

\def\Laplace{\boldsymbol{\Delta}}

\def\bR{{\mathbb{R}}}
\def\bE{{\mathbb{E}}}

\def\cF{{\mathcal{F}}}

\def\oppr{\hbox{\rm \raise.2ex\hbox{${\scriptstyle | }$}\kern-.32em(}}
\def\clbr{\hbox{\rm ]\kern-.15em]}}

\newtheorem{theorem}{Theorem}
\theoremstyle{plain}
\newtheorem{corollary}[theorem]{Corollary}
\newtheorem{definition}[theorem]{Definition}
\newtheorem{proposition}[theorem]{Proposition}

\newtheorem{remark}[theorem]{Remark}
\numberwithin{equation}{section}

\numberwithin{equation}{section}
\numberwithin{theorem}{section}

\begin{document}
\title[Stochastic Porous Medium Equation]
 {A Random Change of Variables and Applications to the
 Stochastic Porous Medium Equation with
Multiplicative Time Noise}
\author{S. V. Lototsky}
\curraddr[S. V. Lototsky]{Department of Mathematics, USC\\
Los Angeles, CA 90089}
\email[S. V. Lototsky]{lototsky@math.usc.edu}
\urladdr{http://math.usc.edu/$\sim$lototsky}

\thanks{The author acknowledges support
 from the NSF CAREER award DMS-0237724.}


\begin{abstract}
A change of variables is introduced to reduce certain nonlinear
stochastic evolution equations with multiplicative noise to the
corresponding deterministic equation. The result is then
used to investigate a stochastic porous
medium equation.

\end{abstract}

\maketitle

\today

\section{Introduction}

Let $U=U(t,x)$ be a solution of the porous medium equation
\begin{equation}
\label{intr00}
U_t=\Laplace U^m,\ t>0,\, x\in \bR, \, m>1,
\end{equation}
 and let $X=X(t)$ be a semi-martingale. Consider  a
stochastic version of \eqref{intr00},
\begin{equation}
\label{intr0}
du=\Laplace u^m dt + u\,dX(t);
\end{equation}
the use of multiplicative noise preserves the positivity
of the solution. The first main result of the paper
is as follows:

\begin{theorem}
\label{th:intr1}
There is a one-to-one correspondence between the
solutions of the deterministic and stochastic
porous medium equations, given by
\begin{equation}
\label{intr1}
u(t,x)=h(t)U(H(t),x),
\end{equation}
where
$$
h(t)=1+\int_0^t h(s)dX(s),\ H(t)=\int_0^th^{m-1}(s)ds.
$$
\end{theorem}

Thus,  many of the results known for the
deterministic equation \eqref{intr00}, as summarized,
 for example, in Aronson \cite{Aronson}, have a clear
 counterpart in  the
stochastic case, and the objective of  the paper is
 to the study this correspondence.
 To keep the presentation
from becoming unnecessarily abstract, a particular semi-martingale
$X$ is considered: $X(t)=\int_0^t g(s)ds + \int_0^t f(s)dw(s)$,
where $w$ is a standard Brownian motion.
The results derived in this manner  can serve as a benchmark for
further investigation of the stochastic porous medium
equation, in particular, driven by space-time noise.

Here is a typical consequence of combining \eqref{intr1}
with the known facts about the deterministic porous medium equation.

\begin{theorem}
Let $u_0(x)$ be a non-negative continuous function with
compact support in $\bR^d$. Then there exists a unique continuous
non-negative random field $u=u(t,x)$
such that, for every smooth compactly supported function $\varphi$,
and every $t>0$, the equality
\begin{equation*}
\begin{split}
\int_{\bR^d} u(t,x)\varphi(x)\,dx&=
\int_{\bR^d} u_0(x)\varphi(x)\,dx+
\int_0^t \int_{\bR^d} u^m(t,x)\Laplace \varphi(x)\, dx\\
&+ \int_0^t \int_{\bR^d} u(t,x)\varphi(x)\, dw(t)
\end{split}
\end{equation*}
holds with probability one (in other words, $u$ is the
solution of $du=\Laplace u^mdt + udw(t)$.)
In addition,
\begin{enumerate}
\item The function $u$ is \Holder continuous in $t,x$
on every set of the form $[T,+\infty)\times \bR^d$, $T>0$;
\item The mean mass is preserved:
$$
\bE \int_{\bR^d} u(t,x)dx=\int_{\bR^d} u(0,x)dx;
$$
\item The support of $u$ is uniformly bounded: there exists a
random variable $\eta$ such that, with probability one,
 $0<\eta<\infty$ and $u(t,x)=0$ for $|x|>\eta$ and all $t>0$.
 \item For every $x\in \bR$, $\lim_{t\to \infty} u(t,x)=0$ with
 probability one.
 \end{enumerate}
 \end{theorem}

Section \ref{sec:CV} presents the general result about the
change of variables \eqref{intr1} for a large class of
nonlinear equations. Section \ref{sec:SPME}
discusses the basic questions about existence, uniqueness,
and regularity of the solution for the stochastic
porous medium equation, and Section \ref{sec:LTB}
investigates the long-time behavior of the solution.
Because of the random time change, the long-time
behaviors in the deterministic and stochastic cases
can be  different.

The following are some comments about the physical origin and
relevance of equation \eqref{intr0}, as well as connections with
 recent work on the
subject.  For a variety of distributed
system, there are two functions that describe the state of the
system at time $t$ and point $x\in \bR^d$;
by the analogy with the classical problem of the
gas flow, it is natural to call these functions the {\tt density}
$\rho(t,x)$ and the {\tt velocity} $\boldsymbol{v}(t,x)$.
A nearly universal {\tt equation of continuity} states that
\begin{equation}
\label{intr3}
\kappa \rho_t+\mathrm{div}(\rho\, \boldsymbol{v})=\sigma,
\end{equation}
where $\sigma=\sigma(t,x,\rho)$
is the density of sources and sinks, and
$\kappa$ is the fraction of the space available to the system.
Moreover, the underlying physical principles dictate the
equation of motion:
\begin{equation}
\label{intr4}
\boldsymbol{v}=\boldsymbol{F}(t,x,\rho, \mathrm{grad} \rho)
\end{equation}
for the known vector-valued function $\boldsymbol{F}$.
An example is {\tt Darcy's Law} for the ideal gas flow,
\begin{equation}
\label{darcy}
\boldsymbol{v}=-\frac{\nu}{\mu}\,\mathrm{grad}\, p,
\end{equation}
where $p=p_0\rho^{\alpha}$, $\alpha>0$,
 is the pressure, $\mu$ is the permeability of the medium, and
 $\nu$ is the viscosity of the gas.

Assume that
 $\boldsymbol{F}(t,x,\rho, \mathrm{grad}\, \rho)=-q(\rho)\mathrm{grad}
\rho-\boldsymbol{a}\psi(\rho)$ for some known functions
$q$ and $\psi$ and a constant vector $\boldsymbol{a}$.
Substituting \eqref{intr4} into \eqref{intr3}, we then get
 \begin{equation}
 \label{intr5}
\kappa \, \rho_t
=\Laplace \Phi(\rho) + \boldsymbol{a}\cdot \mathrm{grad} \Psi(\rho)
+\sigma(t,x,\rho),
\end{equation}
where $\Phi(x)=\int_0^xyq(y)dy$ and $\Psi(x)=x\psi(x)$.
Equation \eqref{intr5} appears in a variety of problems,
including population dynamics and  gas and water propagation.
The underlying physical setting usually requires the
solution $\rho$ of \eqref{intr5} to be non-negative for
all $t>0$ and $x\in \bR$.

The following  two particular cases of \eqref{intr5} are
 worth mentioning:
\begin{enumerate}
\item  Darcy's Law for the ideal gas flow \eqref{darcy}
with constant $\mu,\nu, \kappa$ and  with $\sigma=0$
results in equation \eqref{intr00} with $m=1+\alpha$ for
the  normalized (dimensionless) and suitably
scaled density $u$ of the gas. The main difference between
\eqref{intr00} and  the heat equation ($m=1$)
is that initial  disturbance is propagated by \eqref{intr00}
with finite speed.
\item The same relation \eqref{darcy}
with constant $\mu,\nu$ and   with $\kappa=const$,
 $\sigma(t,x,\rho)=\sigma_0\,\rho$,
$\sigma_0=const.$ leads to a basic model for the
spread of a crowd-avoiding population:
\begin{equation}
\label{intr6}
u_t=\Laplace u^m +\sigma_0\, u,
\end{equation}
where $u$ is a suitably normalized and scaled density of the
population; see Gurtin and MacCamy \cite[Section 2]{Gurtin}.
The number $\sigma_0$ is the reproduction intensity of the
population; the population is birth-dominant if $\sigma_0>0$.
\end{enumerate}

{\em Comparing \eqref{intr0} and \eqref{intr6} we conclude that
\eqref{intr0} can describe the dynamics of  the
crowd-avoiding population with random time-dependent reproduction intensity;}
if $X(t)=\int_0^tg(s)ds+\int_0^tf(s)dw(s)$, then
\begin{equation}
\label{intr66}
\sigma_0=g(t) + f(t)\dot{w}(t),
\end{equation}
where $\dot{w}$ is Gaussian white noise.

In general, allowing the function $\sigma$ in \eqref{intr5}
to be random is  a natural way to introduce randomness in
the porous medium equation. If the positivity of the solution
is not guaranteed with these values of $\sigma$,
  the term $\Laplace u^m$ in the equation  is
 replaced by $\Laplace(|u|^{m-1}u)$. Similar to \eqref{intr0},
the randomness can be in time only, such as
  $\sigma(t,x,u)=u^{\beta}+u\dot{w}$ in Mel$'$nik \cite{Melnik1,
Melnik2}. The randomness can also be in time and space, such as
$\sigma(t,x,u)=\sum_{k}f_k(t,x)\dot{w}_k(t)$ in Kim \cite{Kim},
$\sigma(t,x,u)=F(u)+\dot{W}_Q(t,x)$ in Barbu et al. \cite{Barbu}
and in Da Prato \cite{DaPr}, or
$\sigma(t,x,u)=F(u)+G(u)\dot{W}_Q(t,x)$ in Barbu et al.
\cite{Barbu1}, with
$\dot{W}_Q$ representing Gaussian noise that is white in time but
is sufficiently regular in space. Note that, unlike \eqref{intr0},
none of the above models can benefit from Theorem \ref{th:intr1}.

The derivation of \eqref{intr5} shows that another
way to introduce randomness in the equation
is to allow velocity $\boldsymbol{v}$ to be random,
for example, by  considering a stochastic differential equation
satisfied by $\boldsymbol{v}$. Some results concerning the
resulting porous medium equation were recently  obtained by Sango
\cite{Sango}.

\section{Nonlinear Equations with Multiplicative Noise}
\label{sec:CV}

If $v=v(t,x)$, $t>0$, $x\in \bR^d$
 satisfies the heat equation $v_t=\Laplace v$, and $c$ is a real
 number, then the
 function $u(t,x)=v(t,x)e^{ct}$ satisfies $u_t=\Laplace u+cu$.
 Similarly, if $w=w(t)$ is a standard Brownian motion, then
 $u(t,x)=v(t,x)e^{w(t)-(t/2)}$ satisfies the stochastic \Ito
 equation $du=\Laplace u dt + udw(t)$. The objective of this
 section is to extend these results to some nonlinear equations.

Consider the equation
\begin{equation}
\label{d-nl1}
v_t=F(v,Dv,D^2v,\ldots), \ t>0,  \ x\in \bR^d,
\end{equation}
with some initial condition.
In \eqref{d-nl1}, $v=v(t,x)$ is the unknown function,
$F$ is a give  function,
  $v_t=\partial v/\partial t$ and $D^kv$ denotes a generic
$k$-th order partial derivative of $v$ with respect to $x$.

We also consider the stochastic counterpart of \eqref{d-nl1} for
  the unknown random field $u=u(t,x)$,
 $t>0$, $x\in \bR^d$:
\begin{equation}
\label{snl1}
du=F(u,Du,D^2u,\ldots)dt+u\,(f(t)dw(t)+g(t)dt),
\end{equation}
with the same initial  condition as in \eqref{d-nl1}.
In \eqref{snl1},  $f=f(t)$ is a locally square-integrable deterministic
function, $g$ is a locally integrable deterministic
function,   $w$ is a standard Brownian motion on a stochastic
basis $\mathbb{F}=(\Omega, \cF, \{\cF_t\}_{t\geq 0}, \mathbb{P})$,
and the equation is in the \Ito sense.
We assume that $\mathbb{F}$ satisfies the usual conditions of
completeness of $\cF$ and right continuity of $\cF_t$.

\begin{definition}
Given a stopping time $\tau$,
a  {\tt classical solution} of equation \eqref{snl1}
on the set $\oppr 0,\tau \clbr =\{(t,\omega): t\leq \tau\}$
is a random field $u=u(t,x)$ with the following
properties:
\begin{enumerate}
\item $u$ is continuous in $(t,x)$ for all  $(t,\omega)\in \oppr 0,\tau \clbr$
and $x \in \bR^d$;
\item all the necessary partial derivatives of $u$ with respect to
$x$ exist and are continuous in $(t,x)$ for all
$(t,\omega)\in \oppr 0,\tau \clbr$ and $x \in \bR^d$;
\item The equality
$$
u(t,x)=u(0,x)+\int_0^tF(u(s,x), Du(s,x),\ldots)ds+
\int_0^tu(s,x)(f(s)dw(s)+g(s)ds)
$$
for all $(t,\omega)\in \oppr 0,\tau \clbr$ and $x \in \bR^d$.
 \end{enumerate}
 \end{definition}
Taking in the above definition $f=g=0$, we get the
definition of the classical solution of equation \eqref{d-nl1}.
It turns out that if the function $F$ is homogeneous,
then  there is a one-to-one correspondence
between the classical solutions of \eqref{d-nl1} and \eqref{snl1}.
The key component in this correspondence is a random time change.

\begin{definition}
(a)  We say that the function $F$ is {\tt homogeneous}
{\tt of degree}
$\gamma\geq 1$ if, for every $\lambda>0$,
\begin{equation}
\label{homog}
F(\lambda x, \lambda {y},\lambda {z},\ldots)=\lambda^{\gamma}
F(x,{y},{z},\ldots).
\end{equation}
(b) We say that equation \eqref{d-nl1} is {\tt homogeneous of degree
$\gamma\geq 1$}  if the function $F$ is homogeneous of degree
$\gamma$.
\end{definition}

 \begin{proposition}
 \label{prop:snl1}
 Assume that the function $F$ is homogeneous of degree
 $\gamma$. Define the functions
\begin{equation}
\label{snl2}
h(t)=\exp\left(\int_0^tg(s)ds+
\int_0^tf(s)dw(s)-\frac{1}{2}\int_0^tf^2(s)ds\right),\
 H_{\gamma}(t)=\int_0^th^{\gamma-1}(s)ds.
 \end{equation}
 Then a function  $v=v(t,x)$ is a classical
solution of \eqref{d-nl1} if and only if
\begin{equation}
\label{main}
u(t,x)=v(H_{\gamma}(t),x)h(t)
\end{equation}
 is a  classical solution of \eqref{snl1}.
 \end{proposition}

 \begin{proof} Assume that
  $v$ is a classical solution of  \eqref{d-nl1}. Note that
  $H_{\gamma}'(t)=h^{\gamma-1}(t)$ and
 \begin{equation}
 \label{sde}
 dh(t)=h(t)(f(t)dw(t)+g(t)dt).
 \end{equation}
By the \Ito formula,
$$
du(t,x)=v_t(H_{\gamma}(t),x)h^{\gamma}(t)dt+
u(t,x)(f(t)dw(t)+g(t)dt).
$$
Using \eqref{d-nl1} and homogeneity of $F$,
\begin{equation*}
\begin{split}
v_t(H_{\gamma}(t),x)h^{\gamma}(t)&=
F\Big(h(t)v(H_{\gamma}(t),x),h(t)Dv(H_{\gamma}(t),x),
h(t)D^2v(H_{\gamma}(t),x),\ldots\Big)\\
&=
F(u,Du,D^2u,\ldots),
\end{split}
\end{equation*}
 and therefore $u$ is a classical solution of \eqref{snl1}.

Conversely, assume that $u$ is a classical solution of
\eqref{snl1}. Since $h(t)>0$ for all $t,\omega$,
 the function $H_{\gamma}$
is strictly increasing and  has an inverse function
$R_{\gamma}$.  Define $v(t,x)=z(R_{\gamma}(t))u(R_{\gamma}(t),x)$,
where $z(t)=1/h(t)$. Note that
$$
dz(t)=z(t)\Big(-g(t)dt-f(t)dw(t)+f^2(t)dt\Big).
$$
Then, by the \Ito formula, we conclude that $v$ is a
classical solution of \eqref{d-nl1}.
 \end{proof}

 As a simple illustration of Proposition \ref{prop:snl1}
 consider Burger's equation $v_t=vv_x, \ t>0$ with
 initial condition $v(0,x)=-x$. This
 equation is homogeneous of degree $2$ and has a classical
 solution $v(t,x)=-x/(1+t)$. Then the stochastic equation
 $du=uu_xdt+udw(t)$ with the same initial condition has a solution
 $$
 u(t,x)=-x\,\frac{e^{w(t)-(t/2)}}{1+\int_0^te^{w(s)-(s/2)}ds}.
 $$

 Proposition \ref{prop:snl1} can be generalized in
 several directions:
 \begin{enumerate}
 \item  The functions $f,g$ can be random and adapted.
 In fact, the process $\int_0^tg(s)ds+\int_0^tf(s)dw(s)$
 can be replaced with a semi-martingale $X(t)$, possibly
 with jumps. Then $h(t)$ becomes the stochastic (Dolean)
 exponential of $X$; see Liptser and Shiryaev
 \cite[Section 2.4]{LSh.m}.
 \item Other types of stochastic
 integral and other types
 of random perturbation can be considered, as long as
 the corresponding analog of equation \eqref{sde} can be
 solved and an analogue of the \Ito formula applied.
 For example, consider the fractional Brownian motion
 $B^H$ and take $f=1$, $g=0$. With
  a suitable interpretation of the integral
 $udB^H$ we get $h(t)=e^{B^H(t)-(1/2)t^{2H}}$;
 see, for example, Nualart
 \cite[Chapter 5]{Nualart}.
 \item Since transformation
 \eqref{main} does not involve the space variable $x$,
  Proposition \ref{prop:snl1} also works
   for initial-boundary value problems.
   \item Transformation \eqref{main} can establish a
   one-to-one correspondence between generalized
   (and even viscosity) solutions
   of \eqref{d-nl1} and \eqref{snl1}, but the precise definition of
   the solution and the corresponding arguments in the proof
    require more information about the function $F$.
\end{enumerate}

Let us emphasize that transformation \eqref{main}  does not lead to
a closed-form  solution of the stochastic equation \eqref{snl1}
unless there is a closed-form  solution of the deterministic equation
\eqref{d-nl1}. Some methods of finding closed-form solutions
of nonlinear equations of the type \eqref{d-nl1} are described
in \cite{Estevez}.

\section{Stochastic Porous Medium Equation}
\label{sec:SPME}

Recall that the classical porous medium equation is
\begin{equation}
\label{pme}
U_t(t,x)=\Laplace(U^m(t,x)),\ t>0,
\end{equation}
where $U_t=\partial U/\partial t$ and
$\Laplace$ is the Laplace operator.
This equation can model various physical phenomena for
every $m>0$; in what follows, we will consider only
$m>1$ ($m=1$ is the heat equation). We also assume that
$x\in \bR^d$ so that there are no boundary conditions.
 Note that, without
special restrictions on $m$, the definition of the
solution of \eqref{pme} must include a certain non-negativity
condition on $U$; this condition is also consistent with
the physical interpretation of the solution as a density
of some matter.

The {\tt scaled pressure} $V=V(t,x)$ corresponding to the porous
medium equation \eqref{pme} is defined by
\begin{equation}
\label{pressure}
V(t,x)=\frac{m}{m-1}U^{m-1}(t,x)
\end{equation}
and satisfies
\begin{equation}
\label{pressure1}
V_t=(m-1)V\Laplace V+|\nabla V|^2,
\end{equation}
where $\nabla$ is the gradient.
The function $V$ is extensively used in the study of the
analytic properties of \eqref{pme}.

Let $\mathbb{F}=(\Omega, \cF, \{\cF_t\}_{t\geq 0}, \mathbb{P})$ be a
stochastic basis with the usual assumptions,
$w=w(t)$, a standard Brownian motion on $\mathbb{F}$, and
 $\tau>0$, a stopping time.  Let $f=f(t)$ and $g=g(t)$ be
non-random functions such that $f$ is
locally square-integrable  and
$g$ is locally integrable.

Consider the following equation:
\begin{equation}
\label{spm}
du(t,x)=\Laplace(u^m(t,x))dt+u(t,x)(f(t)dw(t)+g(t)dt),
\ t>0,\ x\in \bR^d.
\end{equation}

\begin{definition} A non-negative,
continuous random field $u=u(t,x)$ is
called a {\tt solution} of equation \eqref{spm}
on the set $\oppr 0,\tau \clbr =\{(t\,\omega): t\leq \tau\}$
if, for every smooth compactly supported function $\varphi=
\varphi(x)$ the following equality holds
for all $(t,\omega)\in\oppr 0,\tau \clbr$:
\begin{equation}
\label{spm-sol}
(u,\varphi)(t)=(u,\varphi)(0)+
\int_0^t (u^m,\Laplace \varphi)(s)ds+\int_0^t (u,\varphi)(s)
(f(s)dw(s)+g(s)ds),
\end{equation}
where
$$
(u,\varphi)(t)=\int_{\bR^d}u(t,x)\varphi(x)dx.
$$
\end{definition}

Note that, with $f(t)=g(t)=0$,  this definition also applies to
the deterministic equation \eqref{pme}.

Define the functions
\begin{equation}
\label{spmH}
\begin{split}
h(t)&=\exp\left(\int_0^tg(s)ds +
\int_0^tf(s)dw(s)-\frac{1}{2}\int_0^tf^2(s)ds\right),\\
\  H(t)&=\int_0^th^{m-1}(s)ds.
\end{split}
 \end{equation}

\begin{proposition}
\label{prop:spm-main}
There is a one-to-one correspondence between the solutions of
\eqref{pme} and \eqref{spm} given by
\begin{equation}
\label{pme-spm}
u(t,x)=U(H(t),x)h(t).
\end{equation}
\end{proposition}

\begin{proof} Note that the porous medium equation
\eqref{pme} is homogeneous of degree $m$. Then \eqref{pme-spm}
is strongly suggested by Proposition \ref{prop:snl1}.
Since the solutions in question are not necessarily
classical, the formal argument involves application of the
\Ito formula in the integral relation \eqref{spm-sol}, but
 is completely analogous to the proof of
Proposition \ref{prop:snl1}.
\end{proof}

 Two immediate consequences of \eqref{pme-spm}
 are the comparison principle and maximum principle for
 equation \eqref{spm}; both follow from the corresponding
 results for the deterministic equation \eqref{pme},
 see the book by V{\'a}zquez \cite[Theorem 9.2]{Vazquez}.

\begin{corollary}
\label{cor:comp}
(a) {\tt Comparison principle:} If $u$, $\widetilde{u}$
are two solutions of \eqref{spm} and $u(0,x)\leq \widetilde{u}(0,x)$
for all $x\in \bR^d$, then $u(t,x)\leq \widetilde{u}(t,x)$ for all
$(t,\omega)\in \oppr 0,\tau \clbr,\, x\in \bR^d$.\\
(b) {\tt Maximum principle:} If $u$ is a solution of \eqref{spm}
and $0\leq u(0,x)\leq M$ for all $x\in \bR^d$, then
$0\leq u(t,x)\leq Mh(t)$ for all
$(t,\omega)\in \oppr 0,\tau \clbr,\, x\in \bR^d$.
\end{corollary}

\begin{remark}
1. In the particular case $f=0$, $g=const.$ the relation
\eqref{pme-spm} was discovered
 by Gurtin and MacCamy \cite{Gurtin}.

 2. While the current paper deals only with non-negative
 solutions of the porous medium equation,
  relation \eqref{pme-spm}
 also holds for all the solutions of
 $U_t=\Laplace(|U|^{m-1}U)$ and $du=\Laplace(|u|^{m-1}u)dt+udw(t)$.
 \end{remark}

It is important to have the definition of the
solution of \eqref{spm} on a random time interval,
because  even a classical solution of \eqref{pme} can
blow up in finite time, and then the corresponding
 solution of \eqref{spm} will blow up in random time.
The {\tt quadratic pressure} solution provides an example.
By direct computation, the function $U^{\rm [qp]}$ defined by
\begin{equation}
\label{qps}
U^{\rm [qp]}(t,x)=\left(\frac{t_1 |x|^2}{t_q-t}\right)^{1/(m-1)},
\ t_q=\frac{m-1}{2mq(2+d(m-1))},\ q>0; \ t_1=t_q|_{q=1},
\end{equation}
is a classical solution of \eqref{pme} for
$(t,x)\in (0,t_q)\times \bR^d$; $U^{\rm [qp]}$ is known as the
{\tt quadratic pressure solution} because the corresponding pressure
is
$$
V^{\rm [qp]}(t,x)=\frac{mt_1|x|^2}{(m-1)(t-t_q)};
$$
for details, see Aronson \cite[pages 3--5]{Aronson} or
V{\'a}zquez \cite[Section 4.5]{Vazquez}.
By Proposition \ref{prop:snl1},
$$
u^{\rm [qp]}(t,x) =U^{\rm [qp]}(H(t),x)h(t)
$$
is a classical solution of \eqref{spm} on
$\oppr 0,\tau_q\clbr$, where
$$
\tau_q=\inf(t>0: H(t)=t_q)
$$
is the blow-up time;
it is certainly possible to have $\tau_q<\infty$ with positive
probability.
On the other hand, if $f=0$, then, with a suitable
choice of $g$, the life of every non-global solution
of \eqref{pme} can be extended
indefinitely: note that
 $g(t)=-\alpha<0$ corresponds to  $H(t)=(1-e^{-\alpha t})/\alpha
<1/\alpha$, and it remains to take $\alpha>0$ sufficiently large.

The following is the main result about global existence, uniqueness,
and regularity of the solution of \eqref{spm}.

\begin{theorem}
\label{th:spme-main}
Assume that the initial condition $u(0,x)$ is non-random and
has the following properties:
\begin{enumerate}
\item non-negative and bounded: $0\leq u(0,x)\leq C$, $x\in \bR^d$;
\item continuous;
\item integrable: $\int\limits_{\bR^d} u(0,x)dx=M,\ 0<M<\infty$;
\item square-integrable: $\int\limits_{\bR^d} u^2(0,x)dx<\infty$.
\end{enumerate}
Then there exists a unique non-negative
 solution $u=u(t,x)$ of \eqref{spm}.
This solution is defined for all $t>0$, $x\in \bR^d$ and
has the following properties:
\begin{enumerate}
\item the mean total mass satisfies
$\bE\int\limits_{\bR^d} u(t,x)dx=Me^{\int_0^tg(s)ds}$, $t>0$;
\item $u(t,\cdot)$ is \Holder continuous (as a function of $x$)
 for every $t>0$;
 \item if $\left(\int_{\bR^d}u^p(0,x)dx\right)^{1/p}=M_p<\infty$, then
 \begin{equation}
 \label{Lp-norm}
 \left(\bE \int_{\bR^d}u^p(t,x)dx\right)^{1/p} \leq
 M_p\exp\left(\int_0^tg(s)ds+\frac{(p-1)}{2}\int_0^tf^2(s)ds\right).
 \end{equation}
\end{enumerate}
In addition, if the functions $f,g$ are
locally bounded, then the function $u$ is \Holder
continuous on $[T,+\infty)\times\bR^d$ for every $T>0$.
\end{theorem}

\begin{proof}
For the deterministic equation, Sabinina \cite{Sabinina} proved
existence, uniqueness, and conservation  of mass (see also the
book by V{\'a}zquez
\cite[Chapter 9]{Vazquez}; Caffarelli and Friedman \cite{CafFr}
proved \Holder continuity. It is also  known \cite[Theorem
9.3]{Vazquez} that
$$
\int_{\bR^d}U^p(t,x)dx\leq \int_{\bR^d}U^p(0,x)dx,\ t>0,\, p>1,
$$
where $U$ is the solution of \eqref{pme}.
 It remains to use Proposition
\ref{prop:spm-main} and note that
$$
\bE h^p(t)=
\exp\left(p\int_0^tg(s)ds+\frac{p(p-1)}{2}\int_0^tf^2(s)ds\right),
\, t>0.
$$ Also, if the functions $f,g$ are locally bounded, then
the function $h$ is \Holder continuous of any order less than
$1/2$.

\end{proof}

Note that the initial condition of the
quadratic pressure solution,
$U^{\rm [qp]}(0,x)=q^{1/(m-1)}|x|^{2/(m-1)}$,
is not bounded. Theorem \ref{th:spme-main} shows that
a blow-up of the solution can be avoided with suitable
growth restrictions on the initial condition. These
restrictions are sufficient but not necessary: consider,
for example
the {\tt linear pressure solution}
$$
u^{\rm [lp]} (t,x)=\left(\frac{m-1}{m}\max
\big(H(t)+x,\,0\big)\right)^{1/(m-1)}h(t),\ t>0,\ x\in \bR.
$$

By analogy with \eqref{pressure},
define the {\tt scaled pressure} corresponding to
equation \eqref{spm} by
\begin{equation}
\label{spressure}
v(t,x)=\frac{m}{m-1} u^{m-1}(t,x).
\end{equation}
An application of the \Ito formula shows that $v$
satisfies
\begin{equation}
\begin{split}
dv&=\left((m-1)v\Laplace v+|\nabla v|^2+\frac{(m-1)(m-2)}{2}\, v f^2
\right)dt\\
&+(m-1)v\, (fdw+gdt).
\end{split}
\end{equation}

On the other hand, equation \eqref{pressure1} is
homogeneous of degree $2$, and so  Proposition
\ref{prop:snl1} suggests an alternative definition:
\begin{equation}
\label{eq:p3}
v(t,x)=V(\widetilde{H}(t),x)\widetilde{h}(t),
\end{equation}
where
\begin{equation*}
\begin{split}
\widetilde{h}(t)&=
\exp\left( (m-1)\int_0^tf(s)dw(s)-\frac{(m-1)^2}{2}
\int_0^tf^2(s)ds\right.\\
&+\left.\frac{(m-1)(m-2)}{2}\int_0^t f^2(s)ds
+(m-1)\int_0^tg(s)ds\right)
\end{split}
\end{equation*}
and $\widetilde{H}(t)=\int_0^t\widetilde{h}(s)ds$.
An observation that $\Big((m-1)(m-2)-(m-1)^2\Big)/2=-(m-1)/2$ shows
that, in fact, $\widetilde{H}(t)=H(t)$ and
$$
V(\widetilde{H}(t),x)\widetilde{h}(t)=\frac{m}{m-1}
U^{m-1} (H(t),x)h^{m-1}(t).
$$
In other words, \eqref{spressure} and \eqref{eq:p3} define the
same function. Another way to phrase this conclusion is
to say that,  for the porous medium equation
\eqref{pme}, the change of variables defined by
Proposition \ref{prop:snl1} commutes with the
transformation  $U\mapsto (m/(m-1))U^{m-1}$.

\section{Barenblatt's solution and long-time behavior}
\label{sec:LTB}

Barenblatt's solution of $U_t=\Laplace(U^m)$ is
\begin{equation}
\label{pme-bt}
U^{\rm [BT]} (t,x;b)=\frac{1}{t^{\alpha}}
\left(\max\left(
b-\frac{{m}-1}{2{m}}\,\beta\,\frac{|x|^2}{t^{2\beta}},\
0\right) \right)^{1/({m}-1)}, \ t>0,\ x\in \bR^d,
\end{equation}
where $b>0$ and
$$
\beta=\frac{1}{({m}-1)d+2},\  \alpha=\beta d.
$$
For the derivation of $U^{\rm [BT]}$ see
Aronson \cite[pages 3--4]{Aronson}, Evans
\cite[Section 4.2.2]{Evans}, or V{\'a}zquez
\cite[Section 4.4.2]{Vazquez}.
 The function has $U^{\rm [BT]}$ the following
properties:
\begin{enumerate}
\item the {\tt total mass} of the solution,
$\int\limits_{\bR^d} U^{\rm [BT]}(t,x;b)dx$, does not depend on $t$
and is uniquely determined by $b$: by direct computation,
\begin{equation}
\label{integral}
\int_{\bR^d} U^{\rm [BT]}(t,x;b)dx
=b^{1/(2\beta(m-1))}
 \left(\frac{{m}-1}{2\pi\,{m}}\,\beta\right)^{-d/2}\,
 \frac{\Gamma\left(\frac{m}{m-1}\right)}
 {\Gamma\left(\frac{m}{m-1}+\frac{d}{2}\right)},
\end{equation}
where $\Gamma$ is the Gamma function
(see also Aronson \cite[Pages 3--4]{Aronson}
and V{\'a}zquez \cite[Section 17.5]{Vazquez});
\item $U^{\rm [BT]}$ is  a classical solution of $U_t=\Laplace(U^m)$
 in the region
\begin{equation}
\label{inter}
\left\{(t,x): |x|\not=\sqrt\frac{2mb}{(m-1)\beta}\  t^{\beta}
\right\};
\end{equation}
\item For every $p,q,t_0>0$, $x_0\in \bR^d$, the function
\begin{equation}
\label{self-sim}
\tilde{U}(t,x)=\left(\frac{p}{q^2}\right)^{1/(m-1)}
U^{\rm [BT]}(pt+t_0,qx+x_0;b)
\end{equation}
is also a solution of \eqref{pme}.
\end{enumerate}
By Proposition \ref{prop:snl1}, the function
\begin{equation}
\label{s-bt}
u^{\rm [BT]}(t,x;b)=U^{\rm [BT]}(H(t),x;b)h(t),
\end{equation}
with $H,h$ given by \eqref{spmH}, is a solution
of the stochastic porous medium equation \eqref{spm}.
In particular,
$$
\bE\int\limits_{\bR^d} u^{\rm [BT]}(t,x;b)dx=\left(\ \int\limits_{\bR^d}
U^{\rm [BT]}(t,x;b)dx\right) e^{\int_0^tg(s)ds}.
$$

Barenblatt's solution $U^{\rm [BT]}$ determines
 the long-time behavior of
{\em every} global solution of the deterministic equation
\eqref{pme}.
Similarly, with  obvious restrictions on $H$,
$u^{\rm [BT]}$ determines the long-time
behavior of the solutions of the stochastic porous medium
equation from Theorem \ref{th:spme-main}.

\begin{theorem}
Assume that $\lim_{t\to infty} H(t)=+\infty$ with probability
one.
Let $u=u(t,x)$ be the solution of the stochastic porous medium
equation constructed under the assumptions of Theorem
 \ref{th:spme-main}. Then, for every $x\in \bR^d$,
 \begin{equation}
 \label{bt-lim1}
 \lim_{t\to \infty} (H(t))^{\beta d}|u(t,x)-u^{\rm [BT]}(t,x;b)|=0
 \end{equation}
 with probability one, where $b$ is such that
 \begin{equation}
\label{integral1}
\int_{\bR^d} u(0,x)dx
 =b^{1/(2\beta(m-1))}
 \left(\frac{{m}-1}{2\pi\,{m}}\,\beta\right)^{-d/2}\,
 \frac{\Gamma\left(\frac{m}{m-1}\right)}
 {\Gamma\left(\frac{m}{m-1}+\frac{d}{2}\right)}.
\end{equation}
 \end{theorem}

 \begin{proof}
 This follows from Proposition \ref{prop:snl1} and
 the corresponding result for the
 deterministic equation,
 $$
\lim_{t\to \infty} t^{\beta d}|U(t,x)-U^{\rm [BT]}(t,x;b)|=0
$$
(see Friedman and Kamin \cite{FK}).
\end{proof}

\begin{remark} Clearly, if $\lim_{t\to \infty} H(t)$ is
finite, then the long-time behavior of the solutions of
\eqref{spm} will be quite different and, in particular,
not as universal. By definition, the function $H$ is
non-decreasing and therefore has a limit (finite or
infinite) with probability one. If $\int_0^{\infty}
f^2(s)ds<\infty$, then $H(t)\uparrow +\infty$ as long as
$\liminf_{t\to \infty} \int_0^t g(s)ds > -\infty$. If
$\int_0^{\infty} f^2(t)dt=+\infty$, then, by the law of
iterated logarithm, $\lim_{t\to \infty}H(t)$
 can be finite with probability one,
and  more information about $f(t)$ and $g(t)$ is
necessary to proceed (see the example below).
\end{remark}

 Further information about the asymptotic behavior of the
 solution can be obtained under additional assumptions about
 the functions $f,g$.

\begin{theorem}
Assume that $\int_0^{\infty} f^2(t)dt=2\sigma^2$ and $\int_0^t
g(t)dt=\mu$ for some $\sigma, \mu\in \bR$. Then, for
every $x\in \bR^d$,
\begin{equation}
\label{s-bt1}
\lim_{t\to \infty}
|u(t,x) -  e^{\xi}\,U^{\rm [BT]}(e^{(m-1)\xi} t,x;b)|=0
\end{equation}
with probability one, where $b$ satisfies
\eqref{integral1} and  $\xi$ is a Gaussian random
variable with mean
$\mu-\sigma^2$ and variance $\sigma^2$.
\end{theorem}

\begin{proof} Under the assumptions of the
theorem, we have, with probability one,
\begin{equation}
\label{conv00}
\lim_{t\to \infty}\left(
\int_0^t g(s)ds+\int_0^tf(s)dw(s)-\frac{1}{2}\int_0^t
f^2(s)ds\right)=\xi.
\end{equation}
Then $\lim_{t\to \infty} h^{m-1}(t)=e^{(m-1)\xi}$ and therefore
$$
\lim_{t\to \infty} \frac{1}{t}\int_0^t h^{(m-1)}(s)ds=e^{(m-1)\xi}.
$$
The result now follows by continuity of the function
$U^{\rm [BT]}$. Note that no rate can be specified in \eqref{s-bt1}
without assumptions about the rate of convergence in
\eqref{conv00}.

\end{proof}

\begin{remark} It follows that equation \eqref{spm} DOES NOT have
a nontrivial invariant measure in any traditional  function space.
Indeed, by \eqref{s-bt1}, if $\lim_{t\to \infty} h(t)$
 exists, then, for every $x\in \bR^d$,
 $$
\lim_{t\to \infty} u(t,x)=\lim_{t\to \infty} U^{\rm [BT]}(t,x;b)=0,
$$
that is, either the solution decays uniformly or the
mass  spreads out to infinity. On the other hand,
if $\lim_{t\to \infty} h(t)$ does not exist, then,
by \eqref{pme-spm}, no non-trivial limit of $u(t,x)$ can
exist either. The same remark applies to the equation  in a bounded
domain.
\end{remark}

Other information about certain  global solutions of
the stochastic porous medium equation can be obtained
by comparison with  Barenblatt's solution.

\begin{theorem}
Assume that the initial condition $u(0,x)$ of \eqref{spm} is
continuous,
non-negative,  and compactly supported in $\bR^d$. Then, with
probability one,
\begin{enumerate}
\item the solution $u(t,x)$ is non-negative and
has compact support in $\bR^d$ for
all $t>0$;
\item the {\tt interface}, that is,
the boundary of the set $\{x\in \bR^d: u(t,x)>0\}$, is moving with
finite speed.
\item  if $\lim_{t\to \infty}H(t)<\infty$, then the
support of the solution remains bounded for all $t>0$.
\end{enumerate}
\end{theorem}

\begin{proof}
By the maximum principle (Corollary
\ref{cor:comp}(b)), if $u(0,x)\geq 0$, then so is $u(t,x)$.
Furthermore, if $u(0,x)=0$ for $|x|>R$ and
 $u(0,x)\leq C$, then, for sufficiently large $C_1$,
$$
u(0,x)\leq \max(C_1-|x|^2,0).
$$
By comparison principle (Corollary
\ref{cor:comp}(a)) we then get $u(t,x)\leq \tilde{U}(H(t),x)h(t)$,
where $\tilde{U}(t,x)$ is a function of the type
\eqref{self-sim} with $p=q=t_0=1$, $x_0=0$, and $b$
sufficiently large. Therefore, $u(t,x)=0$ for
$|x|>C_2H^{\beta}(t)$ for  a suitable (non-random) number $C_2$.
\end{proof}

{\bf Example.} Consider the equation
\begin{equation}
\label{spme-ex}
du=\Laplace u^2\, dt +u\, dw(t)
\end{equation}
and assume that $u(0,x)$ is continuous, non-negative, compactly
supported, and $\int_{\bR^d}u(0,x)dx>0$.
 Then there exists a random variable $\eta$ such that
$0<\eta<\infty$ with probability one and
$u(t,x)=0$ for all $|x|>\eta$ and all $t>0$. Indeed, in this case
$$
h(t)=e^{w(t)-(t/2)},\ H(t)=\int_0^te^{w(s)-(s/2)}ds,
$$
 and, by the previous theorem,
it is enough to show that $H(t)$ is bounded with probability one.
Let $T_0$ be the last time $w(t)$ exceeds $t/4$:
$T_0=\sup\{t>0: w(t)>t/4\}$. Then
$$
H(t)<\int_0^{T_0} e^{w(t)-(t/2)}dt+4e^{T_0/4},\ t>0.
$$
 By the law of iterated logarithm,
$T_0<\infty$ with probability one, and therefore
 $\lim_{t\to \infty}H(t)<\infty$ with probability one.

Notice also that $\lim_{t\to \infty} h(t)=0$ with probability
one, and consequently, for every $x\in \bR^d$,
$$
\lim_{t\to \infty} u(t,x)=\lim_{t\to \infty}h(t)U(H(t),x)=0,
$$
because the solution $U=U(t,x)$ of the deterministic equation
\eqref{pme} with initial condition $U(0,x)=u(0,x)$ is a
uniformly bounded function. On the other hand, we know that
$\int_{\bR^d}U(t,x)dx=\int_{\bR^d} u(0,x)dx$,
and, since $\bE h(t)=1$ for all $t>0$, we conclude that
$\bE \int_{\bR^d} u(t,x)dx=\int_{\bR^d} u(0,x)dx$.
In other words, the solution of the stochastic porous
medium equation \eqref{spme-ex} is supported in the
same (random) compact set for all $t>0$ and decays to
zero as $t\to \infty$, while preserving the mean total
mass.


\begin{thebibliography}{10}

\bibitem{Aronson}
D.~G. Aronson.
\newblock The porous medium equation.
\newblock In {\em Nonlinear diffusion problems (Montecatini Terme, 1985)},
  volume 1224 of {\em Lecture Notes in Math.}, pages 1--46. Springer, Berlin,
  1986.

\bibitem{Barbu}
Viorel Barbu, Vladimir~I. Bogachev, Giuseppe Da~Prato, and Michael R{\"o}ckner.
\newblock Weak solutions to the stochastic porous media equation via
  {K}olmogorov equations: the degenerate case.
\newblock {\em J. Funct. Anal.}, 237(1):54--75, 2006.

\bibitem{Barbu1}
Viorel Barbu, Giuseppe Da~Prato, and Michael R{\"o}ckner.
\newblock Existence and uniqueness of nonnegative solutions to the stochastic
  porous media equation.
\newblock {\em Indiana Univ. Math. J.}, To appear.

\bibitem{CafFr}
Luis~A. Caffarelli and Avner Friedman.
\newblock Regularity of the free boundary of a gas flow in an {$n$}-dimensional
  porous medium.
\newblock {\em Indiana Univ. Math. J.}, 29(3):361--391, 1980.

\bibitem{DaPr}
G.~Da~Prato, Michael R{\"o}ckner, B.~L. Rozovskii, and Feng-Yu Wang.
\newblock Strong solutions of stochastic generalized porous media equations:
  existence, uniqueness, and ergodicity.
\newblock {\em Comm. Partial Differential Equations}, 31(1-3):277--291, 2006.

\bibitem{Estevez}
P.~G. Est{\'e}vez, Changzheng Qu, and Shunli Zhang.
\newblock Separation of variables of a generalized porous medium equation with
  nonlinear source.
\newblock {\em J. Math. Anal. Appl.}, 275(1):44--59, 2002.

\bibitem{Evans}
Lawrence~C. Evans.
\newblock {\em Partial differential equations}, volume~19 of {\em Graduate
  Studies in Mathematics}.
\newblock American Mathematical Society, Providence, RI, 1998.

\bibitem{FK}
Avner Friedman and Shoshana Kamin.
\newblock The asymptotic behavior of gas in an {$n$}-dimensional porous medium.
\newblock {\em Trans. Amer. Math. Soc.}, 262(2):551--563, 1980.

\bibitem{Gurtin}
Morton~E. Gurtin and Richard~C. MacCamy.
\newblock On the diffusion of biological populations.
\newblock {\em Math. Biosci.}, 33(1-2):35--49, 1977.

\bibitem{Kim}
Jong~Uhn Kim.
\newblock On the stochastic porous medium equation.
\newblock {\em J. Differential Equations}, 220(1):163--194, 2006.

\bibitem{LSh.m}
R.~Sh. Liptser and A.~N. Shiryayev.
\newblock {\em Theory of martingales}, volume~49 of {\em Mathematics and its
  Applications (Soviet Series)}.
\newblock Kluwer Academic Publishers Group, Dordrecht, 1989.

\bibitem{Melnik2}
S.~A. Mel{\cprime}nik.
\newblock The dynamics of solutions of the {C}auchy problem for a stochastic
  equation of parabolic type with power non-linearities.
\newblock {\em Teor. \u Imov\=\i r. Mat. Stat.}, (64):110--117, 2001.

\bibitem{Melnik1}
S.~A. Mel{\cprime}nik.
\newblock Estimates for the solution of the {C}auchy problem for a stochastic
  differential equation of parabolic type with power nonlinearities (a strong
  source).
\newblock {\em Differ. Uravn.}, 38(6):802--808, 862, 2002.

\bibitem{Nualart}
D.~Nualart.
\newblock {\em The {M}alliavin calculus and related topics}.
\newblock Probability and its Applications. Springer-Verlag, Berlin, second
  edition, 2006.

\bibitem{Sabinina}
E.~S. Sabinina.
\newblock On the {C}auchy problem for the equation of nonstationary gas
  filtration in several space variables.
\newblock {\em Soviet Math. Dokl.}, 2:166--169, 1961.

\bibitem{Sango}
M.~Sango.
\newblock Weak solutions for a doubly degenerate quasilinear parabolic equation
  with random forcing.
\newblock {\em Discrete Contin. Dyn. Syst. Ser. B}, 7(4):885--905 (electronic),
  2007.

\bibitem{Vazquez}
Juan~Luis V{\'a}zquez.
\newblock {\em The porous medium equation}.
\newblock Oxford Mathematical Monographs. The Clarendon Press Oxford University
  Press, Oxford, 2007.

\end{thebibliography}

\def\cprime{$'$} \def\cprime{$'$}

\end{document}